\documentclass[11pt]{amsart}

\usepackage{amsmath}
\usepackage{amssymb}
\usepackage{amsthm}

\usepackage{times}
\usepackage[utf8]{inputenc}
\usepackage[T1]{fontenc}
\usepackage{tikz-cd}

\newtheorem{thm}{Theorem}[section]
\newtheorem{cor}[thm]{Corollary}
\newtheorem{lem}[thm]{Lemma}
\newtheorem*{thm*}{Theorem}

\newtheorem{prob}[thm]{Question}

\theoremstyle{definition}
\newtheorem{remark}[thm]{Remark}
\newtheorem{defin}[thm]{Definition}

\DeclareMathOperator{\pr}{\leq_{pr}}

\DeclareMathOperator{\B}{\mathcal B}
\DeclareMathOperator{\K}{\mathcal K}
\DeclareMathOperator{\F}{\mathbb{F}}
\DeclareMathOperator{\N}{\mathbb{N}}
\DeclareMathOperator{\Z}{\mathbb{Z}}
\DeclareMathOperator{\G}{\mathbb{G}}

\DeclareMathOperator{\E}{\mathcal E}

\newcommand{\rg}{\operatorname{rg}}
\newcommand{\dom}{\operatorname{dom}}

\newcommand{\dd}{\mathfrak{d}}
\newcommand{\aaa}{\mathfrak{a}}

\title{Pure-homogeneous Abelian groups}
\author{Ziemowit Kostana}
\address{Institute of Mathematics of the Czech Academy of Sciences, \v{Z}itn\'{a} 25, 115~67 Prague 1, Czech Republic}
\address{Wroc\l{}aw University of Science and Technology, Wybrze\.{z}e Wyspia\'nskiego 27, 50~370 Wroc\l{}aw, Poland}
\email{ziemowit.kostana@pwr.edu.pl}

\begin{document}

\begin{abstract}
We study Fra\"iss\'e classes of Abelian groups with pure embeddings. We characterize Abelian groups that are universal and homogeneous for: finitely co-generated groups, finite groups, groups of size less that $\kappa$, where $\kappa$ is strongly inaccessible.
\end{abstract}

\maketitle



\section{Introduction}
\subsection{Fra\"iss\'e-J\'onsson theory}

The classical Fra\"iss\'e theory studies relations between certain classes of finite structures, called \emph{Fra\"iss\'e classes}, and countable structures that are homogeneous with respect to finite substructures. However, since the original paper of Fra\"iss\'e \cite{fraisse}, the theory received a large number of far-reaching generalizations. The category-theoretic framework for Fra\"iss\'e theory was introduced by Droste and G\"obel \cite{dg1} \cite{dg2}, and was further generalized by Kubi\'s \cite{kub}. The theory originally developed in the language of models has became an influential technology in general topology \cite{solecki}, topological dynamics \cite{kpt}, theory of Polish groups \cite{truss} \cite{kpt}, or functional analysis \cite{garbulinska} \cite{espana}. The article \cite{macpherson} is an excellent survey of what is known about homogeneous structures and their automorphism groups, and \cite{kechris-rosendal} contains an extensive introduction to the study of topological properties of these groups.

The theory of uncountable Fra\"iss\'e limits is often called \emph{Fra\"iss\'e-J\'onsson theory}, as J\'onsson was the first to generalize the original work of Fra\"iss\'e to the uncountable case \cite{jonsson}. Perhaps the most remarkable instance of Fra\"iss\'e-J\'onsson theory is Parovi\v{c}enko Theorem, stating that under Continuum Hypothesis $\mathcal{P}(\omega)/\operatorname{Fin}$ is the unique Boolean algebra of size continuum that has Strong Countable Separation Property \cite{parovicenko}.

A streamlined exposition of the classical Fra\"iss\'e theory can be found in \cite{hodges}.

\subsection{Summary of results}

We study the Fra\"iss\'e classes of Abelian groups, with pure embeddings. Our main results are:
\begin{enumerate}
    \item[] 
    \begin{thm*}[A]
	The class of all finitely co-generated groups, together with pure embeddings, is a Fra\"iss\'e class. In particular, there exists a unique up to isomorphism group $\F_1$ that is pure-injective for finitely co-generated subgroups, and is an increasing union of countably many finitely co-generated pure subgroups. This group can be represented as
    $$\F_1 = \bigoplus_{j \in \mathbb N}\bigoplus_{p \in \mathbb P}\bigoplus_{n \in \mathbb N\cup\{\infty\}}\Z(p^n).$$
    \end{thm*}
    
    \item[] 
    \begin{thm*}[B]
	The class of all finite groups, together with pure embeddings is a Fra\"iss\'e class. In particular, there exists a unique up to isomorphism group $\F_0$ that is pure-injective for finite subgroups, and is an increasing union of countably many finite pure subgroups. This group can be represented as
    $$\F_0 := \bigoplus_{j \in \mathbb N}\bigoplus_{p \in \mathbb P}\bigoplus_{n \in \mathbb N}\Z(p^n).$$
    \end{thm*}

    \item[]
    \begin{thm*}[C]
	Let $\kappa$ be a strongly inaccessible cardinal, and assume that the group $\F$ has the following property:  whenever $C\leq \F$ has size less than $\kappa$, then $C$ is contained in an algebraically compact direct summand $D \pr \F$ such $\F / D$ contains a pure copy of $\G_\lambda$, for every $\lambda{<}\kappa$. Then $\F$ is pure-injective for the class of all groups of size ${<}\kappa$.
    \end{thm*}

    \item[]
    \begin{thm*}[D]
    Let $\kappa$ be a strongly inaccessible cardinal. The class of all groups if size ${<}\kappa$, together with pure embeddings, is a Fra\"iss\'e class. In particular, there exists a unique up to isomorphism group of size $\kappa$ -- denoted $\F_\kappa$ -- that is pure-injective for subgroups of size ${<}\kappa$. 
    
    The group $\F_\kappa$ is isomorphic to the ${<}\kappa$-supported product of length $\kappa$ of all co-cyclic groups, where the isomorphism type of every co-cyclic group appears $\kappa$ many times. 
    \end{thm*}
    
\end{enumerate}

\section{Preliminaries}

All groups are assumed to be Abelian. Every Abelian group $A$ is a module over the ring of integers, so we use the expressions $k \cdot a$, of $k\cdot A$, for $a \in A, \; k\in \Z$ in the usual sense. The notation is mostly standard. The symbol $\mathbb P$ denotes the set of all primes. 

We collect here known facts from group theory and Fra\"iss\'e-J\'onsson thoery. All results from this section -- except Theorem \ref{ctblycogenerated} -- are proven in \cite{fuchs}.

\subsection{Pure subgroups}

\begin{defin}
    A subgroup $A\leq B$ is a \emph{pure subgroup} if for every $a \in A$, every integer $k\geq 1$, if there exists $b \in B$ such that $a=k\cdot b$, then there also exists $a' \in A$ such that $a=k\cdot a'$. We then write $A \pr B$. 
    
    A homomorphism $\phi:A\longrightarrow B$ is a \emph{pure embedding} if it is a monomorphism onto a pure subgroup.
\end{defin}

This notion lies strictly between a subgroup and a direct summand. In the realm of finite (even finitely co-generated) groups -- these notions coincide:

\begin{lem}[Corollary 27.6, \cite{fuchs}]\label{finitedirectsummand}
    A finitely co-generated pure subgroup is always a direct summand.
\end{lem}

\begin{lem}[Corollary 28.3, \cite{fuchs}]\label{finitedirectsummand2}
    If $A$ is a pure subgroup of $B$ and $B/A$ is finitely generated, then $A$ is a direct summand of $B$.
\end{lem}

The following is probably the simplest scenario in which the notions differ.
\begin{lem}
    The infinite direct sum $\bigoplus \limits_{i<\omega}\Z$ is a pure subgroup of $\prod\limits_{i<\omega}\Z$ that is not a direct summand.
\end{lem}
\begin{proof}
    Purity is straightforward, so assume towards a contradiction that $\bigoplus\limits_{i<\omega}\Z$ has a complementary direct summand $C \leq \prod\limits_{i<\omega}\Z$, so we have a decomposition
    $$\prod\limits_{i<\omega}\Z = \big(\bigoplus\limits_{i<\omega}\Z \big) \oplus C.$$
    In particular, we have
    $$C \simeq \prod\limits_{i<\omega}\Z / \bigoplus\limits_{i<\omega}\Z.$$

    The product $\prod\limits_{i<\omega}\Z$ has the property that every non-zero element has only finitely many divisors. However, it is easy to check that in $C$ the element
    $$(1!,2!,\ldots) + \bigoplus\limits_{i<\omega}\Z \in\prod\limits_{i<\omega}\Z / \bigoplus\limits_{i<\omega}\Z$$
    is divisible by every natural number. Hence $C$ cannot be a direct summand in $\prod\limits_{i<\omega}\Z$.
\end{proof}

\begin{remark}
    The same example shows that a union of an increasing chain of direct summands need not be a direct summand: Every group$\prod\limits_{i<n}\Z$ has a complement in $\prod\limits_{i<\omega}\Z$, but the union
    $$\bigcup\limits_{n<\omega}\prod\limits_{i<n}\Z = \bigoplus\limits_{i<\omega}\Z$$
    does not.
\end{remark}

The latter property makes pure embeddings more robust for Fra\"iss\'e theory, as an increasing union of pure subgroups is evidently a pure subgroup. We state two more Lemmas that will be useful.

\begin{lem}[Proposition 26.2, \cite{fuchs}]\label{extending}
Every finite subgroup can be extended to a countable pure subgroup, and every infinite pure subgroup can be extended to a pure subgroup of the same cardinality.
\end{lem}

\begin{lem}[Lemma 26.1, \cite{fuchs}]\label{prufer}
Let $B$ and $C$ be subgroups of $A$ such that $C\leq B \leq A$. Then
\begin{enumerate}
    \item If $C \pr B$ and $B \pr A$, then $C\pr A$.
    \item If $B \pr A$, then $B/C \pr A/C$.
    \item If $C \pr A$, and $B/C \pr A/C$, then $B \pr A$.
\end{enumerate}
    
\end{lem}

\subsection{Finitely co-generated groups}

The notion of a generating set has a natural dualization:

\begin{defin}
    A subset $L$ of a group $A$ is a \emph{set of co-generators} if for every group $B$, every homomorphism $\phi:A\longrightarrow B$, the equality $L \cap \ker(\phi) = 0$ implies that $\phi$ is a monomorphism. A group $A$ is \emph{finitely co-generated} if it admits a finite set of co-generators.
\end{defin}

A special case of this is

\begin{defin}
    A group $A$ is \emph{co-cyclic} if there exists $a \in A$ such that $\{a\}$ is a set of co-generators.
\end{defin}

Similarly to the cyclic groups, the isomorphism types of co-cyclic groups admit a simple characterization:

\begin{thm}[Theorem 3.1 \cite{fuchs}]\label{co-cyclic}

A group $A$ is co-cyclic if and only if it is of the form $\Z(p^n)$, for some $p \in \mathbb P$ and $n \in \N\cup \{ {\infty} \}$.
\end{thm}

Co-cyclic groups appear naturally as quotients by maximal subgroups.

\begin{lem}[Proposition 25.2, \cite{fuchs}]\label{maximalquotient}
    Let $a \in A$ be non-zero, and $M\leq A$ be a maximal (in the sense of inclusion) subgroup of $A$ excluding $a$. Then $A/M$ is co-cyclic.
\end{lem}

Finitely co-generated groups also have very transparent structure.

\begin{thm}[Theorem 25.1, \cite{fuchs}]\label{co-generated}
A group is finitely co-generated if and only if it is a direct sum of finitely many co-cyclic groups.
\end{thm}

We will see in the next subsection, that being $\kappa$-co-generated trivializes on the first infinite cardinality.

\subsection{Algebraically compact groups}

It is well-known that every divisible group is a direct summand of every group containing it as a subgroup. Here we restrict this property only to pure subgroups.

\begin{defin}
    A group $A$ is \emph{algebraically compact} if it is a direct summand of every group $G$ containing $A$ as a pure subgroup.
\end{defin}

\begin{thm}[Theorem 38.1, \cite{fuchs}]
    A group $A$ is algebraically compact if and only if it is a direct summand of a direct product of co-cyclic groups.
\end{thm}

\subsection{Divisible and pure-injective hulls}

\begin{defin}
    A subgroup $A\leq G$ is \emph{essential} is there is no non-zero subgroup $B\leq G$ such that $B\cap A = 0$.
\end{defin}

The following theorem combines a few results from \cite{fuchs} about the divisible hull.

\subsubsection{Divisible Hull}

\begin{thm}[Kulikov, \cite{fuchs}]\label{divisiblehull}
    Every group $A$ admits an extension $A \leq \dd(A)$ such that $\dd(A)$ is divisible, and has the following universal property: Whenever $f:A \longrightarrow D$ is an embedding of $A$ into a divisible group $D$, there exists a unique embedding $\overline{f}:\dd(A) \longrightarrow D$ such that $\overline{f} \restriction A = f$.

    \begin{center}
		\begin{tikzcd}[row sep=1.5cm, column sep=1.5cm]
		\dd(A)\arrow[r, dashed, "\overline{f}"]
		& D
	    \\ A \arrow[u, "\leq"]\arrow[ur, swap, "f"]
		\end{tikzcd}
	\end{center}

    The operator $\dd(-)$, called the \emph{divisible hull} has the following properties:
    \begin{enumerate}
        \item The extension $A \leq \dd(A)$ is unique up to isomorphism over $A$.
        \item If $A$ is infinite, then $|\dd(A)|=|A|$.
        \item $\dd(A)$ is a unique -- up to isomorphism over $A$ -- divisible extension of $A$ such that $A$ is an essential subgroup of $\dd(A)$.
        
    \end{enumerate}

\end{thm}

\subsubsection{Countably co-generated groups}

\begin{thm}\label{ctblycogenerated}
    A group is countably co-generated if and only if it is countable.
\end{thm}
\begin{proof}
    The whole group constitutes a set of co-generators, so the non-trivial direction is that a countably co-generated groups cannot be uncountable. 
    
    Let $L$ be a countable set of co-generators of a group $G$. Then $L$ also co-generates $\dd(G)$: suppose that $\phi:\dd(G)\longrightarrow B$ satisfies $\ker(\phi) \cap L = 0$. Then $\phi \restriction G$ is a monomorphism, so $\ker(\phi) \cap G = 0$. Given that $G$ is essential subgroup of $\dd(G)$, it follows that $\ker(\phi)=0$.

    Therefore, by replacing $G$ with $\dd(G)$, we can assume that $G$ is divisible. Suppose first, that $\langle L \rangle \leq G$ is not essential. Then there is non-zero $B\leq G$ that trivially intersects $\langle L \rangle$. Consider the natural surjection:
    $$\pi:G \longrightarrow G/B.$$
    We have $\ker(\pi) \cap L = 0$, but evidently $\pi$ is not a monomorphism, contradicting that $L$ co-generates $G$.

    Suppose, on the other hand, that $\langle L \rangle$ is essential in $G$. Then $G = \dd(\langle L \rangle)$, and we have $|G|=|\langle L \rangle | =\omega$.
\end{proof}

\subsubsection{Algebraic Hull}

\begin{thm}[Maranda, \cite{fuchs}] \label{pureinjectivehull}
    Every group $G$ admits an extension $G \leq \aaa(G)$ such that $\aaa(G)$ is algebraically compact, and has the following universal property: Whenever $f:G \longrightarrow F$ is a pure embedding of $G$ into an algebraically compact group $F$, there exists a unique pure embedding $\overline{f}:\aaa(G) \longrightarrow F$ such that $\overline{f} \restriction G = f$.

    \begin{center}
		\begin{tikzcd}[row sep=1.5cm, column sep=1.5cm]
		\aaa(G)\arrow[r, dashed, "\overline{f}"]
		& F
	    \\ G \arrow[u, "\pr"]\arrow[ur, swap, "f"]
		\end{tikzcd}
	\end{center}

        In particular, the extension $G \leq \aaa(G)$ is unique up to isomorphism over $G$.
\end{thm}

Theorem \ref{divisiblehull} has a parallel for pure emebddings.

\subsection{Fra\"iss\'e-J\'onsson classes}

We consider a class $\K$, consisting of structures in some first-order language, together with a distinguished class of embeddings, denoted by $\E$. We assume that $\E$ contains all isomorphisms and is closed under compositions. For $A,B \in \K$ we write $A\sqsubseteq_{\E}B$ if $A\subseteq B$ and the inclusion map belongs to $\E$.

\begin{defin}
    For a limit ordinal $\delta$, an $\sqsubseteq_{\E}$-increasing sequence $\{A_\alpha \mid \alpha<\delta\}$ of $\K$-structures is \emph{continuous} if for every limit $\beta<\delta$ we have
    $$A_\beta=\bigcup\limits_{\alpha<\beta}A_\alpha.$$
    
\end{defin}

\begin{defin}
    For an uncountable cardinal $\kappa$, the class $(\K,\E)$ is \emph{${<}\kappa$-closed} if whenever $\{A_\alpha \mid \alpha<\delta\}$ is a continuous sequence of $\K$-structures, and $\delta{<}\kappa$, then

    $$A_\delta:=\bigcup_{\alpha<\delta}A_\alpha \in \K,$$
    and $A_\alpha \sqsubseteq_{\E} A_\delta$ for every $\alpha<\delta$.    
\end{defin}

\begin{defin} For a class $(\K,\E)$ we will say that:
	\begin{itemize}
		\item $(\K,\E)$ has the \emph{Joint Embedding Property} ($\operatorname{JEP}$), if for each $a,b \in \mathcal{K}$ there exists $\E$-embeddings of $a$ into $c$ and of $b$ into $c$, for some $c \in \K$.\\
	\begin{center}
	\begin{tikzcd}
		& a  \ar[dr]
		&
		& \\
		& 
		&c
		&
		&
		&
		\\
		& b  \ar[ur]
		&
		& 
	\end{tikzcd}
	\end{center}
		\item $(\K,\E)$ has the \emph{Amalgamation Property} ($\operatorname{AP}$), if for each pair of embeddings $f:a\longrightarrow b$, $g:a\longrightarrow c$ from $\E$ there exists $d\in \K$, together with a pair of embeddings $f':b\longrightarrow d$, $g':c\longrightarrow d$ in $\E$, such that $f'\circ f = g'\circ g$.\\
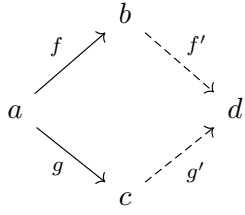
		\begin{center}
		\begin{tikzcd}
			& b  \ar[dr,dashed, "f'"]
			&
			& \\
			a \ar[ur, "f"] \ar[dr, swap, "g"] 
			& 
			&d 
			&
			&
			&
			\\
			& c  \ar[ur, dashed, swap, "g'"]
			&
			& 
		\end{tikzcd}
		\end{center}
	\end{itemize}
\end{defin}

\begin{defin}A structure $\mathbb A$ is:
	\begin{itemize}
		\item \emph{$(\K,\E)$-universal}, if for every $B \in \K$, there exists an $\E$-embedding $B \longrightarrow \mathbb A$,
		\item \emph{$(\K,\E)$-injective}, if for any
        $\E$-embedding $f:B\longrightarrow \mathbb A$, any $C \in \K$ such that $B\sqsubseteq_{\E} C$, there exists an extension of $f$ to an $\E$-embedding $\overline{f}:C\longrightarrow \mathbb A$.
        
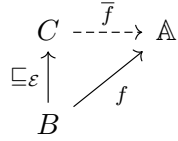
		\begin{center}
		\begin{tikzcd}
		C\arrow[r, dashed, "\overline{f}"]
		& \mathbb A
	    \\B\arrow[u, "\sqsubseteq_{\E}"]\arrow[ur, swap, "f"]
		\end{tikzcd}
	    \end{center}
		
		\item \emph{$(\K,\E)$-homogeneous}, if any isomorphism between substructures $B_0,B_1 \sqsubseteq_{\E} \mathbb A$ can be extended to an automorphism of $\mathbb A$.
        \begin{center}
		
		\begin{tikzcd}
			B_0 \arrow[d, swap, "\sqsubseteq_{\E}"]  \arrow[r, "\simeq"]
			& B_1 \arrow[d, "\sqsubseteq_{\E}"] 
			\\ \mathbb A \arrow[r, dashed]
			& \mathbb A
		\end{tikzcd}
		
	\end{center}

	\end{itemize}
\end{defin}

\begin{thm}[Fra\"iss\'e Theorem, \cite{fraisse},\cite{jonsson},\cite{kub}]\label{fraissethm}
Let $\lambda$ be an infinite cardinal satisfying $\lambda^{{<}\lambda}=\lambda$. Let $\K$ be a class of structures in a finite first-order language, and let $\E$ be a distinguished class of embeddings, closed under compositions, and containing all isomorphisms. Let $\K_\lambda$ be a sub-class of $\K$, consisting only of models of size less than $\lambda$. Assume that $(\K_\lambda,\E)$ satisfies the following properties:
\begin{enumerate}
    \item Joint Embedding Property,
    \item Amalgamation Property,
    \item Is closed under $\E$-substructures,
    \item Is ${<}\lambda$-closed.
\end{enumerate}

    Then there exists a structure $\mathbb A_\lambda$ that satisfies
    \begin{enumerate}
        \item[a)] $\mathbb A_\lambda=\bigcup\limits_{\alpha{<}\lambda}A_\alpha$,
        where 
            \begin{itemize}
                \item $A_\alpha \in \K_\lambda$, for every $\alpha{<}\lambda$,
                \item $A_\alpha \sqsubseteq_{\E} A_\beta$, whenever $\alpha<\beta{<}\lambda$,
                \item $A_\beta=\bigcup\limits_{\alpha<\beta}A_\alpha$, whenever $\beta{<}\lambda$ is limit.
            \end{itemize}
         \item[b)] $\mathbb A_\lambda$ is $(\K_\lambda,\E)$-universal,
         \item[c)] $\mathbb A_\lambda$ is $(\K_\lambda,\E)$-injective.
    \end{enumerate}
    Moreover, the properties $1.-4.$ properties determine the structure uniquely $\mathbb A_\lambda$ up to isomorphism, and any structure satisfying them is also $(\K_\lambda,\E)$-homogeneous.
\end{thm}

The Fra\"iss\'e Theorem motivated the important notion of a Fra\"iss\'e class,

\begin{defin}\label{fraisseclass}
    The class $(\K_\lambda,\E)$ satisfying conditions $(1)-(4)$ is called the \emph{Fra\"iss\'e class} of the length $\lambda$. The structure $\mathbb A_\lambda$ is its \emph{Fra\"iss\'e limit}.
\end{defin}

We proceed to the proof.

\begin{proof}
Enumerate as $\{A_\alpha \mid \alpha{<}\lambda \}$ all (up to isomorphism) structures from $\K_\lambda$. We aim to recursively build an $\sqsubseteq_{\E}$-increasing sequence of structures $F_\alpha \in \K_\lambda$, and then set :
    $$\mathbb A_\lambda=\bigcup_{\alpha{<}\lambda}F_\alpha.$$
    For bookkeeping purposes, we fix a partition $\{\Phi_\gamma \mid \gamma{<}\lambda \}$ of the cardinal $\lambda$, consisting of sets of cardinality $\lambda$, such that $\min{\Phi_\gamma}\ge \gamma$, for all $\gamma < \lambda$.
	\vspace{1em}
    
	\textbf{Existence:} Let $F_0=A_0$, and enumerate as $\{g_\gamma \mid \gamma \in \Phi_0 \}$ all $\E$-embeddings $g$, such that $\dom{g}\sqsubseteq_{\E} F_0$ and $\rg{g} \sqsubseteq_{\E} A_\gamma$ for some $\gamma{<}\lambda$.
	
	\begin{itemize}
	\item  In a successor step $F_\alpha$ is defined, and so is the set $\{g_\gamma \mid \gamma \in \Phi_\alpha \}$. Note that if $\beta$ satisfies $\alpha \in \Phi_\beta$, then $\beta \leq \alpha$, so in particular, $g_\alpha$ is already defined, and $\dom{g_\alpha}\subseteq F_\alpha$. Using the Amalgamation Property, we can find $h$, and $F'_{\alpha+1}\in \K_\lambda$, closing the diagram
	\begin{center}
		
		\begin{tikzcd}
			\dom{g_\alpha} \arrow[d, swap, "g_\alpha"]  \arrow[r,"\sqsubseteq_{\E}"]
			& F_\alpha \arrow[d,dashed] 
			\\ \rg{g_\alpha} \arrow[r, "h", dashed]
			& F'_{\alpha+1}
		\end{tikzcd}
		
	\end{center}
	
	Using the Joint Embedding Property we can enlarge $F'_{\alpha+1}$ to $F_{\alpha+1}$, so that $F_{\alpha+1}$ contains an isomorphic copy of $A_\alpha$. 
    
    We use the set $\Phi_{\alpha+1}$ to index all $\E$-embeddings starting from substructures of $F_{\alpha+1}$.
	
	\item In the limit step, we just set 
    $$F_\alpha=\bigcup_{\gamma<\alpha}F_\gamma,$$
    and use the set $\Phi_{\alpha}$ to index all $\E$-embeddings starting from $\sqsubseteq_{\E}$-substructures of $F_{\alpha}$.
	
	\end{itemize}
	
	Universality of the $\mathbb A_\lambda$ is straightforward, so we leave it for the reader, and proceed to show the injectivity.

    Fix $B\sqsubseteq_{\E} C \in \K_\lambda$, and an $\E$-embedding $i:B\longrightarrow \mathbb A_\lambda$. There is some $\delta{<}\lambda$ such that $g_\delta = i^{-1}:i[B]\longrightarrow B\sqsubseteq_{\E} C$. The recursive construction gives us the commutative diagram: 

    \begin{center}
		
		\begin{tikzcd}
			i[B] \arrow[d, swap, "i^{-1}"]  \arrow[r, "\sqsubseteq_{\E}"]
			& F_\delta \arrow[d, "\sqsubseteq_{\E}"] 
			\\ C \arrow[r, "h", dashed]
			& F'_{\delta+1}
		\end{tikzcd}
		
	\end{center}

    It follows that $h$ and $i$ agree on $B$.
    \vspace{1em}
    
    \textbf{Uniqueness:}
    Suppose that we have another structure $\mathbb A'_\lambda$ satisfying the conclusion of the theorem. In particular there exists a decomposition
    $$\mathbb A'_\lambda = \bigcup_{\alpha{<}\lambda}{A'_\alpha},$$
    into an $\E$-increasing sequence of $\K_\lambda$-structures.
    Let $f_0:A_0\longrightarrow A'_0$ be given. We will recursively define an increasing sequence of $\E$-embeddings $f_\alpha$ such that 
    $$f=\bigcup_{\alpha{<}\lambda}f_\alpha$$ will be an isomorphism from $\mathbb A_\lambda$ onto $\mathbb A'_\lambda$.
	\begin{itemize}
		\item Suppose $\alpha$ is an even successor ordinal and $f_\alpha$ is defined. The cardinal $\lambda$ is regular, so there exists $\beta>\alpha$ such that $\dom{f_\alpha}\sqsubseteq_{\E} A_\beta$. Since $\mathbb A'$ is $(\K_\lambda, \E)$-injective, we can find $f_{\alpha+1}$ closing the diagram
		
		\begin{center}
			\begin{tikzcd}[column sep=4 pt]
			\dom{f_\alpha} \arrow[d,swap, "f_\alpha"] 
			& \sqsubseteq_{\E} 
			& A_\beta \arrow[d,dashed, "f_{\alpha+1}"] 
			\\ \rg{f_\alpha} 
			& \sqsubseteq_{\E}
			& \mathbb{A}'_\lambda
		\end{tikzcd}
		\end{center}
		In particular, $A_{\alpha+1} \subseteq \dom(f_{\alpha+1})$.
		\item Suppose $\alpha$ is an odd successor ordinal and $f_\alpha$ is defined. There exists $\beta>\alpha$ such that $\rg{f_\alpha}\sqsubseteq_{\E} B_\beta$. Since $\mathbb A$ is $(\K_\lambda,\E)$-injective, we can find $f_{\alpha+1}$ closing the diagram
		
		\begin{center}
		\begin{tikzcd}[column sep=4 pt]
			\dom{f_\alpha} 
			& \sqsubseteq_{\E} 
			& \mathbb A_\lambda 
			\\ \rg{f_\alpha} \arrow[u, "f_\alpha^{-1}"] 
			& \sqsubseteq_{\E}
			& A'_\beta \arrow[u,dashed,swap, "f_{\alpha+1}^{-1}"] 
		\end{tikzcd}
		\end{center}
		In particular, $A'_{\alpha+1} \subseteq \rg(f_{\alpha+1})$.
		\item If $\alpha$ is limit, we set 
        $$f_\alpha=\bigcup_{\gamma<\alpha}f_\gamma.$$
		
	\end{itemize}
	
	It is evident that $f=\displaystyle{\bigcup_{\alpha{<}\lambda}f_\alpha}$ is an isomorphism.

    \textbf{Homogeneity:} To see that any structure satisfying the conditions $1.-4.$ is $(\K_\lambda,\E)$-homogeneous, note that in the proof of the uniqueness we started the construction from an arbitrary partial isomorphism $f_0$. Therefore we can just take $\mathbb A'_\lambda = \mathbb A_\lambda$ and follow the same construction.
\end{proof}

\begin{defin}
    The structure $\mathbb A$ \emph{pure-universal}/ \emph{pure-injective}/ \emph{pure-homogeneous} of a class $\K$ it it is $(\K_\lambda,\E)$-universal/ $(\K_\lambda,\E)$-injective/ $(\K_\lambda,\E)$-homogeneous for the class $\E$ of all pure embeddings.
\end{defin}

\begin{remark}
    There is a slight conflict of terminology here. In algebra \emph{injective} typically means that that every homomorphism from a small structure can be extended to a homomorphism from a bigger structure. In particular, the notion of pure-injectivity in \cite{fuchs} follows this convention: a group is pure-injective if every partial homomorphism defined on a pure subgroup can be extended to a homomorphism defined on a full group. We are considering only monomorphisms, so our notion of pure-injectivity is different than in \cite{fuchs}.

\end{remark}

\section{Finite and finitely co-generated groups}

\begin{thm} \label{mainthm1}
    The class of all finitely co-generated groups, together with pure embeddings, is a Fra\"iss\'e class. In particular, there exists a unique up to isomorphism group $\F_1$ that is pure-injective for finitely co-generated subgroups, and is an increasing union of countably many finitely co-generated pure subgroups. This group can be represented as
    $$\F_1 = \bigoplus_{j \in \mathbb N}\bigoplus_{p \in \mathbb P}\bigoplus_{n \in \mathbb N\cup\{\infty\}}\Z(p^n).$$
\end{thm}

\begin{proof}
    Straight from the definition, the group $\F_1$ is a pure, direct limit of finitely co-generated groups. By Theorem \ref{fraissethm}, it is sufficient to show that $\F_1$ is pure-injective for finitely co-generated groups. 

    Let $A$ be any finitely co-generated group. By Lemma \ref{finitedirectsummand} we know that a pure embedding of a finitely co-generated group is an embedding onto a direct summand. Therefore the relevant diagram is, in fact, of the form:

    \begin{center}
        \vspace{1cm}
		\begin{tikzcd}[row sep=1.5cm]
		A\arrow[d, swap, "\pr"]\arrow[r, "\phi"]
		& \phi[A]\oplus\F'_1 = \F_1
        \\ A\oplus C
		\end{tikzcd}
        \vspace{1cm}
	\end{center}

    We need to find a diagonal arrow by embedding $C$ onto a direct summand of $\F'_1$. For this, it is sufficient to show that $\F_1' \simeq \F_1$.

    We have a decomposition
    $$\F_1 = \F_0 \oplus \mathbb D,$$
    where 
    $$\F_0 = \bigoplus_{j \in \mathbb N}\bigoplus_{p \in \mathbb P}\bigoplus_{n \in \mathbb N}\Z(p^n)$$
    and 
    $$\mathbb D = \bigoplus_{j \in \mathbb N}\bigoplus_{p \in \mathbb P}\Z(p^\infty)$$
    In particular, $\mathbb D$ is divisible. By Theorem \ref{co-generated}, the group $A$ is a direct sum of co-cyclic groups, so it admits a similar decomposition: 
    $$A = F \oplus D,$$
    where $F$ is finite and $D$ is divisible. The image $\phi[D]$ is also divisible, so we have $\phi[D]\leq \mathbb D$. From the structure of divisible groups (Theorem 23.1 \cite{fuchs}), we can easily conclude that $\mathbb D / \phi[D] \simeq \mathbb D$, so it follows that 
    $$(\F_0 \oplus \mathbb D)/\phi[A] \simeq ((\F_0 \oplus \mathbb D)/\phi[D])/\phi[F] \simeq (\F_0 \oplus \mathbb D)/\phi[F].$$
    But given that $F$ is finite, we have 
    $$(\F_0 \oplus \mathbb D)/\phi[F] \simeq \F_0 \oplus \mathbb D = \F_1.$$

\end{proof}

\begin{remark}
The consequence of Theorem \ref{mainthm1} is that $\F_1$ contains a copy of every countable group that is a countable union of finite pure subgroups. However, this class is rather limited. The group $\F_1$ does not have an element of infinite order, so it doesn't have a copy of $\Z$. 
\end{remark}

The very same proof gives also the variant for finite groups. The only difference is that the limit group does not contain quasi-cyclic $\Z(p^\infty)$ factors.

\begin{thm} \label{mainthm2}
	The class of all finite groups, together with pure embeddings is a Fra\"iss\'e class. In particular, there exists a unique up to isomorphism group $\F_0$ that is pure-injective for finite subgroups, and is an increasing union of countably many finite pure subgroups. This group can be represented as
    $$\F_0 := \bigoplus_{j \in \mathbb N}\bigoplus_{p \in \mathbb P}\bigoplus_{n \in \mathbb N}\Z(p^n).$$
\end{thm}

\section{Infinite groups}

\subsection{Universal groups}

As we have seen, $\F_0$ is not universal for countable subgroups. Here we present an example of a group that is pure-universal for all countable groups. It follows from Theorem 41.9 \cite{fuchs} that it necessarily has to be uncountable. The proof is a minor modification of the proof of Lemma 30.3 \cite{fuchs}.

\begin{thm}[essentially Lemma 30.3, \cite{fuchs}]\label{universal}
    Let $\lambda$ be an infinite cardinal, and let $\G_\lambda$ be the product of co-cyclic groups, where each isomorphism type of a co-cyclic group appears $\lambda$ many times. Then every group of size $\lambda$ purely embeds into $\G_\lambda$.
\end{thm}

\begin{proof}
    Let $A$ be an arbitrary group of size $\lambda$. For each $a \in A$, and $k \geq 0$, let $M^k_a$ be a maximal subgroup of $A$ disjoint with $\{a \}$, and extending $k\cdot A$ if such group exists, and $M_a^k=A$ otherwise (recall that $0\cdot A = 0$). Let
    $$\pi: A \longrightarrow \prod_{\substack{a \in A \\  k \geq 0}}A/M_a^k$$
    be the product of natural quotient maps. 
    
    Observe that $\pi$ is 1-1: if $a \in A$ is non-zero, then $M^0_a \leq A$ is a proper subgroup, so corresponding coordinate witnesses $\psi(a) \neq 0$.

    Moreover, $\pi$ is a pure embedding: suppose that $k >1$ and $a \in A$ is not divisible by $k$. Then $a \notin k\cdot A$, and so $M^k_a$ is a proper subgroup of $A$. The coset $a+M^k_a$ is not divisible by $k$ in $A/M^k_a$, so in particular the image $\pi(a)$ is not divisible by $k$.

    Finally, by Lemma \ref{maximalquotient} each quotient $A/M_a^k$ is co-cyclic, so the product 
    $$\prod\limits_{\substack{a \in A \\ k \geq 0}}A/M_a^k$$
    is a direct summand in $\G_\lambda$.
\end{proof}
\subsection{Infinite pure-injective groups}
\begin{thm} \label{mainthm3}
	Let $\kappa$ be a strongly inaccessible cardinal, and assume that the group $\F$ has the following property:  whenever $C\leq \F$ has size less than $\kappa$, then $C$ is contained in an algebraically compact direct summand $D \pr \F$ such $\F / D$ contains a pure copy of $\G_\lambda$, for every $\lambda{<}\kappa$. Then $\F$ is pure-injective for the class of all groups of size ${<}\kappa$.
\end{thm}

\begin{proof}
Consider a diagram:

    \begin{center}
        \vspace{1cm}
		\begin{tikzcd}[row sep=1cm]
		A\arrow[d, swap, "\pr"]\arrow[r, "\phi"]
		& \F
        \\ B
		\end{tikzcd}
        \vspace{1cm}
	\end{center}
where $\phi$ is a pure embedding, and $|B|=\lambda{<}\kappa$. By Theorem \ref{universal}, we can replace $B$ with $\G_\lambda$ (which has size $2^\lambda{<}\kappa$). By the property of $\F$, there is a decomposition $\F = \mathbb D \oplus \mathbb E$ where $\mathbb D$ is algebraically compact, $\phi[A]\pr \mathbb D$, and $\mathbb E$ contains a pure copy of $\G_\lambda$. Let $\psi$ be the extension of $\phi$ given by Theorem \ref{pureinjectivehull}.

There is a decomposition $\G_\lambda = \aaa(A) \oplus C$, and a pure embedding
$\eta:C \longrightarrow \mathbb E$, so we can further extend $\psi$ to

$$\psi + \eta : \aaa(A)\oplus C \longrightarrow \mathbb D \oplus \mathbb E = \F.$$

    \begin{center}
        \vspace{1cm}
		\begin{tikzcd}[row sep=1.5cm, column sep=1.5cm]
		A\arrow[d, swap, "\pr"]\arrow[r, "\phi"]
		& \mathbb D \arrow[dd, "\pr"]
        \\ \aaa(A)\arrow[d, swap, "\pr"]\arrow[ur, swap, "\psi"]
        \\ \G_\lambda = \aaa(A)\oplus C \arrow[r, "\psi + \eta"] & \mathbb D \oplus \mathbb E = \F
        
		\end{tikzcd}
        \vspace{1cm}
	\end{center}
\end{proof}

\begin{thm}\label{mainthm4}

    Let $\kappa$ be a strongly inaccessible cardinal. The class of all groups if size ${<}\kappa$, together with pure embeddings, is a Fra\"iss\'e class. In particular, there exists a unique up to isomorphism group of size $\kappa$ -- denoted $\F_\kappa$ -- that is pure-injective for subgroups of size ${<}\kappa$. 
    
    The group $\F_\kappa$ is isomorphic to the ${<}\kappa$-supported product of length $\kappa$ of all co-cyclic groups, where the isomorphism type of every co-cyclic group appears $\kappa$ many times, or -- equivalently -- the direct union
    
    $$\F_\kappa := \bigcup\limits_{\lambda{<}\kappa} \G_\lambda,$$ where $\G_\lambda$ is defined in Theorem \ref{universal}. 
    
    The group $\F_\kappa$ is pure-homogeneous for subgroups of size ${<}\kappa$, and pure-universal for the groups of size $\kappa$. 
\end{thm}

\begin{proof}
    The group $\F_\kappa$ clearly satisfies the hypothesis of Theorem \ref{mainthm3}, so it pure-injectve for groups of size ${<}\kappa$. By Lemma \ref{extending}, the regularity of $\kappa$ ensures that any group of size $\kappa$ is a union of a pure-increasing chain of pure subgroups of size ${<}\kappa$. Now Theorem \ref{fraissethm} ensures that $\F_\lambda$ is the Fra\"iss\'e limit of the class of all groups of size ${<}\kappa$, with pure embeddings. 
\end{proof}

A direct consequence of pure-homogeneity of $\F_\kappa$ is

\begin{cor}
    If $\kappa$ is strongly inaccessible, that every pure subgroup of $\F_\kappa$ if size ${<}\kappa$ is a direct summand.
\end{cor}

It seems plausible that unique pure-homogeneous groups exit for much broader variety of cardinalities. The standard assumption in Fra\"iss\'e-J\'onsson theory is $\kappa=\kappa^{<\kappa}$, so the natural question that comes to mind is the following:

\begin{prob}
    Is it possible to relax the assumption of $\kappa$ being strongly inaccessible in Theorem \ref{mainthm4}?
\end{prob}

\section*{Acknowledgements}
The author was supported by the GA\v{C}R project EXPRO 20-31529X and RVO: 67985840.

\end{document}